\documentclass[12pt,a4paper]{article}
\usepackage{amsfonts}
\usepackage{epsfig}
\usepackage{latexsym}
\usepackage{amsthm}
\usepackage{amsmath}
\usepackage{amssymb}
\usepackage{array}
\usepackage{enumerate}
\usepackage{graphicx,graphics}
\newtheorem{theorem}{Theorem}[section]
\newtheorem{lemma}[theorem]{Lemma}
\newtheorem{proposition}[theorem]{Proposition}
\newtheorem{example}[theorem]{Example}
\newtheorem{corollary}[theorem]{Corollary}
\newtheorem{definition}[theorem]{Definition}

\begin{document}

\title{\bf Laplacian eigenvalues of the zero divisor graph of the ring $\mathbb{Z}_{n}$}
\author{Sriparna Chattopadhyay\footnote{Supported by SERB NPDF scheme (File No. PDF/2017/000908), Department of Science and Technology, Government of India} \and Kamal Lochan Patra \and Binod Kumar Sahoo}
\maketitle
\begin{abstract}
We study the Laplacian eigenvalues of the zero divisor graph $\Gamma\left(\mathbb{Z}_{n}\right)$ of the ring $\mathbb{Z}_{n}$ and prove that $\Gamma\left(\mathbb{Z}_{p^t}\right)$ is Laplacian integral for every prime $p$ and positive integer $t\geq 2$. We also prove that the Laplacian spectral radius and the algebraic connectivity of $\Gamma\left(\mathbb{Z}_{n}\right)$ for most of the values of $n$ are, respectively, the largest and the second smallest eigenvalues of the vertex weighted Laplacian matrix of a graph which is defined on the set of proper divisors of $n$. The values of $n$ for which algebraic connectivity and vertex connectivity of $\Gamma\left(\mathbb{Z}_{n}\right)$ coincide are also characterized.\\

\end{abstract}

\section{Introduction}

Let $G$ be a finite simple graph with vertex set $V(G)=\{v_{1},v_{2},\ldots,v_{n}\}$. For $1\leq i\neq j\leq n$, we write $v_i\sim v_j$ if $v_i$ is adjacent to $v_j$ in $G$. The {\it adjacency matrix} of $G$ is the $n\times n$ matrix $A(G)=(a_{ij})$, where $a_{ij}=1$ or $0$ according as $v_i\sim v_j$ in $G$ or not. The {\it Laplacian matrix} $L(G)$ of $G$ is defined by $L(G):=D(G)-A(G)$, where $D(G)$ is the diagonal matrix of vertex degrees of $G$. The eigenvalues of $L(G)$ are called the {\it Laplacian eigenvalues} of $G$. Since $L(G)$ is a real, symmetric and positive semidefinite matrix, all its eigenvalues are real and nonnegative. Since the sum of the entries in each row of $L(G)$ is zero, the smallest eigenvalue of $L(G)$ is $0$ with corresponding eigenvector $\textbf{1}=[1,1,\ldots,1]^T$. The second smallest eigenvalue of $L(G)$, denoted by $\mu(G)$, is called the {\it algebraic connectivity} of $G$. Applying the Perron-Frobenius theorem to the matrix $(n-1)I-L(G)$, it follows that $\mu(G)$ is positive if and only if $G$ is connected. The largest eigenvalue of $L(G)$, denoted by $\lambda(G)$, is called the {\it Laplacian spectral radius} of $G$. Fiedler proved that $\lambda(G) = n -\mu(\overline{G})$ \cite[3.7($1^{\circ}$)]{fei}, where $\overline{G}$ denotes the complement graph of $G$. Characteristic polynomial of $L(G)$ is called the {\it Laplacian characteristic polynomial} of $G$.

The graph $G$ is called {\it Laplacian integral} if all the Laplacian eigenvalues of $G$ are integers. The {\it vertex connectivity} of $G$, denoted by $\kappa(G)$, is the minimum number of vertices which need to be removed from $V(G)$ so that the induced subgraph of $G$ on the remaining vertices is disconnected or has only one vertex. Fiedler proved that $\mu(G)\leq \kappa (G)$ for a noncomplete graph $G$ \cite[4.1]{fei}. For a complete graph $K_m$ on $m$ vertices, $\mu(G)=m=\kappa(G)+1$.

The {\it spectrum} of a square matrix $B$, denoted by $\sigma(B)$, is the multiset of all the eigenvalues of $B$. If $\mu_{1},\mu_{2},\ldots, \mu_{t}$  are the distinct eigenvalues of $B$ with respective multiplicities $m_{1},m_{2},\ldots, m_{t}$, then we shall denote the spectrum of $B$ by
$$\sigma(B)=\left\{\begin{matrix}
                \mu_{1} & \mu_{2} & \cdots & \mu_{t} \\
                m_{1} & m_{2} & \cdots & m_{t} \\
                \end{matrix}\right\}.$$
For a graph $G$, the spectrum of $L(G)$ is called the \textit{Laplacian spectrum} of $G$, which is denoted by $\sigma _{L}(G)$. The Laplaian spectrum of graphs have been widely studied in the literature, see \cite{mohar} and the references therein.

Let $R$ be a commutative ring with multiplicative identity $1\neq 0$. A nonzero element $x\in R$ is called a {\it zero divisor} of $R$ if there exist a nonzero element $y\in R$ such that $xy=0$. The notion of zero divisor graph of a commutative ring was first introduced by I. Beck in \cite{beck} and it was later modified by Anderson and Livingston in \cite{anderson} as the following. The {\it zero divisor graph} $\Gamma(R)$ of $R$ is the simple graph with vertex set consisting of the zero divisors of $R$, in which two distinct vertices $x$ and $y$ are adjacent if and only if $xy=0$. Note that $\Gamma(R)$ is the empty graph (that is, no vertex) if $R$ is an integral domain.

For a positive integer $n$, let $\mathbb{Z}_{n}$ denote the ring of integers modulo $n$. Different aspects of the zero divisor graph $\Gamma\left(\mathbb{Z}_n\right)$ of $\mathbb{Z}_{n}$ are studied in \cite{akbari, anderson, ju, mathew}. In this paper, we study the Laplacian eigenvalues of the zero divisor graph $\Gamma\left(\mathbb{Z}_{n}\right)$. In Section \ref{sec-srtucture}, we study the structure of $\Gamma\left(\mathbb{Z}_{n}\right)$ and prove that $\Gamma\left(\mathbb{Z}_{n}\right)$ is a generalized join of certain complete graphs and null graphs\footnote{By a {\it null graph} we mean a graph with no edges.}. In Section \ref{sec-L-spectrum}, we discuss the Laplacian spectrum of $\Gamma\left(\mathbb{Z}_{n}\right)$. In Section \ref{sec-integrality}, we prove that the graph $\Gamma\left(\mathbb{Z}_{p^t}\right)$ is Laplacian integral for every prime $p$ and positive integer $t\geq 2$. Finally, in Section \ref{sec-alg-con}, we study the algebraic connectivity and Laplacian spectral radius of $\Gamma\left(\mathbb{Z}_{n}\right)$. We characterize the values of $n$ for which algebraic connectivity and vertex connectivity of  $\Gamma\left(\mathbb{Z}_{n}\right)$ coincide. We also prove that the Laplacian spectral radius and the algebraic connectivity of $\Gamma\left(\mathbb{Z}_{n}\right)$ for most of the values of $n$ are, respectively, the largest and the second smallest eigenvalues of the vertex weighted Laplacian matrix of a graph which is defined on the set of proper divisors of $n$.

\section{$\Gamma(\mathbb{Z}_{n})$ as a generalized join graph}\label{sec-srtucture}

\subsection{Generalized join graphs}

For two graphs $G_{1}$ and $G_{2}$ with disjoint vertex sets, recall that the {\it join} $G_{1}\vee G_{2}$ of $G_{1}$ and $G_{2}$ is the graph obtained from the union of $G_{1}$ and $G_{2}$ by adding new edges from each vertex of $G_{1}$ to every vertex of $G_{2}$. The following is a generalization of the definition of join graph (which is called generalized composition graph in \cite{schwenk}).

\begin{definition}\label{gen-join}
Let $G$ be a graph on $k$ vertices with $V(G)=\{v_1,v_2,\ldots ,v_k\}$ and let $H_{1}, H_{2}, \ldots , H_{k}$ be $k$ pairwise disjoint graphs. The $G$-generalized join graph $G[H_{1}, H_{2}, \ldots , H_{k}]$ of $H_{1}, H_{2}, \ldots , H_{k}$ is the graph formed by replacing each vertex $v_i$ of $G$ by the graph $H_{i}$ and then joining each vertex of $H_{i}$ to every vertex of $H_{j}$ whenever $v_i\sim v_j$ in $G$.
\end{definition}

Note that if $G$ consists of two adjacent vertices only, then the $G$-generalized join graph $G[H_{1}, H_{2}]$ coincides with the usual join $H_{1}\vee H_{2}$ of $H_{1}$ and $H_{2}$.
The following lemma is useful for us.

\begin{lemma}\label{connected}
Let $G$ be a graph with $V(G)=\{v_1,v_2,\ldots ,v_k\}$ and let $H_{1}, H_{2}, \ldots , H_{k}$ be $k$ pairwise disjoint graphs. If $G$-generalized join graph $G[H_{1}, H_{2}, \ldots , H_{k}]$ is connected, then $G$ is connected. Conversely, if $k\geq 2$ and $G$ is connected, then $G[H_{1}, H_{2}, \ldots , H_{k}]$ is connected.
\end{lemma}
\begin{proof}
Suppose that $k\geq 2$ and $G$ is connected. Let $x$ and $y$ be two distinct vertices of $ G[H_{1}, \ldots , H_{k}]$ with $x\in V(H_{i})$ and $y\in V(H_{j})$. First assume that $i\neq j$. Then $v_i\neq v_j$. Let $v_i\sim v_{i_{1}}\sim v_{i_{2}}\sim \cdots \sim v_{i_{l}}\sim v_j$ be a path between $v_i$ and $v_j$ in $G$. Take a vertex $a_{i_{r}}\in V(H_{i_{r}})$ for $1\leq r\leq l$. Then $x\sim a_{i_{1}}\sim a_{i_{2}}\sim \cdots \sim a_{i_{l}}\sim y$ is a path between $x$ and $y$ in $G[H_{1}, H_{2}, \ldots , H_{k}]$. Now assume that $i=j$. Then $x,y\in V(H_{i})$. If $H_{i}$ is connected, then there is nothing to prove. Otherwise, since $k\geq 2$, consider a neighbour $v_l$ of $v_i$ in $G$. Then, for $a\in V(H_{l})$, $x\sim a\sim y$ is a path in $G[H_{1}, H_{2}, \ldots , H_{k}]$. So $G[H_{1}, H_{2}, \ldots , H_{k}]$ is connected.

Conversely, assume that $G[H_{1}, H_{2}, \ldots , H_{k}]$ is connected. Let $v_s$ and $v_t$ be two distinct vertices of $G$ (so $s\neq t$). Take $x\in V(H_{s})$ and $y\in V(H_{t})$. Let $x=x_{1}\sim x_{2}\sim x_{3}\sim \cdots \sim x_{l-1}\sim x_{l}=y$ be a shortest path between $x$ and $y$ in $G[H_{1}, H_{2}, \ldots , H_{k}]$. Then observe that no two vertices of $x=x_{1}, x_{2}, \ldots ,x_{l-1}, x_{l}=y$ are in the same vertex set $V(H_r)$ for any $r\in\{1,2,\ldots,k\}$. For $2\leq j\leq l-1$, assuming that $x_j\in V(H_{i_j})$ for some $i_j\in \{1,2,\ldots,k\}\setminus\{s,t\}$, we can see that $v_s\sim v_{i_{2}}\sim v_{i_{3}}\sim \cdots \sim v_{i_{l-1}}\sim v_t$ is a path between $v_s$ and $v_t$ in $G$. So $G$ is connected.
\end{proof}

\subsection{Structure of $\Gamma(\mathbb{Z}_{n})$}

For two integers $s, t$, the greatest common divisor of $s$ and $t$ is denoted by $(s,t)$. Throughout the paper, we denote the elements of the ring $\mathbb{Z}_{n}$ by $0,1,2,\cdots ,n-1$. To avoid triviality of $\Gamma(\mathbb{Z}_{n})$ being an empty graph, we assume that $n\neq 1$ and that $\mathbb{Z}_{n}$ is not an integral domain. So $n\geq 4$ and $n$ is not a prime. A nonzero element $x$ of $\mathbb{Z}_n$ is called a {\it unit} if $xy=1$ for some element $y\in \mathbb{Z}_n$. Any nonzero element $a$ of $\mathbb{Z}_{n}$ is either a unit or a zero divisor according as $(a,n)=1$ or not.  The number of vertices in $\Gamma(\mathbb{Z}_{n})$ is $n-\phi (n)-1$, where $\phi$ is the Euler's totient function.

An integer $d$ is called a {\it proper divisor} of $n$ if $1<d<n$ and $d\mid n$. Let $d_{1},d_{2},\cdots ,d_{k}$ be the distinct proper divisors of $n$. For $1\leq i\leq k$, we define the following sets:
$$A_{d_i}=\{ x\in\mathbb{Z}_{n}: (x,n)=d_i\}.$$
The sets $A_{d_1},A_{d_2},\ldots, A_{d_k}$ are pairwise disjoint and we can partition the vertex set of $\Gamma(\mathbb{Z}_{n})$ as
$$V(\Gamma(\mathbb{Z}_{n}))=A_{d_{1}}\cup A_{d_{2}}\cup \ldots \cup A_{d_{k}}.$$
The following result is proved in \cite[Proposition 2.1]{mathew}.

\begin{lemma}\cite{mathew}\label{mathew}
$|A_{d_i}|=\phi\left(\frac{n}{d_i}\right)$ for $1\leq i\leq k$.
\end{lemma}

Note that any element $x$ of $A_{d_i}$ can be written as $x=m_xd_i$ for some integer $m_x$ with $0<m_x<\frac{n}{d_i}$ and $(m_x,\frac{n}{d_i})=1$. The following lemma describes adjacency of vertices in $\Gamma(\mathbb{Z}_{n})$.

\begin{lemma}\label{struc}
For $i,j\in\{1,2,\ldots, k\}$, a vertex of $A_{d_i}$ is adjacent to a vertex of $A_{d_j}$ in $\Gamma(\mathbb{Z}_{n})$ if and only if $n$ divides $d_id_j$.
\end{lemma}

\begin{proof}
Let $x\in A_{d_i}$ and $y\in A_{d_j}$. Then $x=m_xd_i$ and $y=m_yd_j$ for some integers $m_x,m_y$ with $0<m_x<\frac{n}{d_i}$, $0<m_y<\frac{n}{d_j}$ and $(m_x,\frac{n}{d_i})=1=(m_y,\frac{n}{d_j})$. The vertices $x$ and $y$ are adjacent in $\Gamma(\mathbb{Z}_{n})$ if and only if $n$ divides $xy$, that is, if and only if $n$ divides $m_xm_yd_id_j$.
Since $(m_x,\frac{n}{d_i})=1=(m_y,\frac{n}{d_j})$, we have
$$n|m_xm_yd_id_j\Leftrightarrow \frac{n}{d_i}|m_xm_yd_j\Leftrightarrow\frac{n}{d_i}|m_yd_j\Leftrightarrow n|m_yd_id_j\Leftrightarrow \frac{n}{d_j}|m_yd_i\Leftrightarrow\frac{n}{d_j}|d_i\Leftrightarrow n|d_id_j.$$
This completes the proof.
\end{proof}

As a consequence of Lemmas \ref{mathew} and \ref{struc}, we have the following.

\begin{corollary}\label{struc-coro}
The following hold:
\begin{enumerate}
\item[(i)] For $i\in\{1,2,\ldots, k\}$, the induced subgraph $\Gamma(A_{d_i})$ of $\Gamma(\mathbb{Z}_{n})$ on the vertex set $A_{d_i}$ is either the complete graph $K_{\phi\left(\frac{n}{d_i}\right)}$ or its complement graph $\overline{K}_{\phi\left(\frac{n}{d_i}\right)}$. Indeed, $\Gamma(A_{d_i})$ is $K_{\phi\left(\frac{n}{d_i}\right)}$ if and only if $n$ divides $d_i^2$.
\item[(ii)] For $i,j\in\{1,2,\ldots, k\}$ with $i\neq j$, a vertex of $A_{d_{i}}$ is adjacent to either all or none of the vertices of $A_{d_{j}}$ in $\Gamma(\mathbb{Z}_{n})$.
\end{enumerate}
\end{corollary}

The above corollary implies that the partition $A_{d_{1}}\cup A_{d_{2}}\cup \cdots \cup A_{d_{k}}$ of the vertex set $V(\Gamma(\mathbb{Z}_{n}))$ of $\Gamma(\mathbb{Z}_{n})$ is an {\it equitable partition} \cite[p.83]{cvet}, that is, every vertex in $A_{d_i}$ has the same number of neighbors in $A_{d_j}$ for all $i,j \in \{1,2,\ldots , k\}$.

Denote by $\Upsilon_n$ the simple graph with vertices the proper divisors $d_1,d_2,\ldots ,d_k$ of $n$, in which two distinct vertices $d_i$ and $d_j$ are adjacent if and only if $n$ divides $d_id_j$.
If $n=p_1^{n_1}p_2^{n_2}\cdots p_r^{n_r}$ is the prime power factorization of $n$, where $r, n_1,n_2,\ldots, n_r$ are positive integers and $p_1,p_2,\ldots,p_r$ are distinct prime numbers, then the number of vertices of $\Upsilon_n$ is given by:
$$|V(\Upsilon_n)|=\underset{i=1}{\overset{r}\prod} (n_i+1)-2.$$
The graph $\Upsilon_n$ shall play an important role in the rest of the paper.

\begin{lemma}\label{UP}
$\Upsilon_n$ is a connected graph.
\end{lemma}

\begin{proof}
Consider two vertices $d_i$ and $d_j$ of $\Upsilon_n$ with $i\neq j$ and let $(d_i,d_j)=l$. If $l\neq 1$, then $\frac{n}{l}$ is a vertex of $\Upsilon_n$ and $d_i\sim\frac{n}{l}\sim d_j$ in $\Upsilon_n$. If $l=1$, then $d_id_j$ divides $n$ and so $d_i\sim\frac{n}{d_i}\sim\frac{n}{d_j}\sim d_j$ in $\Upsilon_n$. So $\Upsilon_n$ is connected.
\end{proof}

The following lemma says that $\Gamma(\mathbb{Z}_{n})$ is a generalized join of certain complete graphs and null graphs.

\begin{lemma}\label{divisorgraph}
Let $\Gamma(A_{d_{i}})$ be the induced subgraph of $\Gamma(\mathbb{Z}_{n})$ on the vertex set $A_{d_i}$ for $1\leq i\leq k$. Then
$\Gamma(\mathbb{Z}_{n})=\Upsilon_n[\Gamma(A_{d_{1}}),\Gamma(A_{d_{2}}),\cdots ,\Gamma(A_{d_{k}})].$
\end{lemma}

\begin{proof}
Replace the vertex $d_i$ of $\Upsilon_n$ by $\Gamma(A_{d_{i}})$ for $1\leq i\leq k$. Then the result can be seen using Lemma \ref{struc}.
\end{proof}

\begin{corollary}
$\Gamma(\mathbb{Z}_{n})$ is connected.
\end{corollary}

\begin{proof}
If $n$ has at least two proper divisors, then $\vert V(\Upsilon_n)\vert\geq 2$ and so the corollary follows from Lemmas \ref{connected}, \ref{UP} and \ref{divisorgraph}. If $n$ has exactly one proper divisor, then $n=p^2$ for some prime $p$. In this case, $\Gamma\left(\mathbb{Z}_{p^2}\right)=\Gamma(A_p)$ has $p-1$ vertices and $\Gamma(A_p)$ is a complete graph by Corollary \ref{struc-coro}(i).
\end{proof}

We note that the above corollary also follows from a more general result by Anderson and Livingston in \cite[Theorem 2.3]{anderson} which says that the zero divisor graph of any commutative ring with multiplicative identity is connected.

\begin{corollary}\label{z-n-complete}
$\Gamma(\mathbb{Z}_{n})$ is a complete graph if and only if $n=p^2$ for some prime $p$.
\end{corollary}

\begin{proof}
If $n=p^2$ for some prime $p$, then $\Gamma(\mathbb{Z}_{n})=\Gamma(A_p)$ is the complete graph $K_{\phi(p)}$ by Corollary \ref{struc-coro}(i). Conversely, assume that $\Gamma(\mathbb{Z}_{n})$ is a complete graph. If $p$ is a prime divisor of $n$, then $\Gamma(A_p)$ must be a complete graph. So $n\mid p^2$ by Corollary \ref{struc-coro}(i) and it follows that $n=p^2$.
\end{proof}

\begin{example}
The zero divisor graph $\Gamma(\mathbb{Z}_{18})$ of $\mathbb{Z}_{18}$ is shown in Figure 1. Here $V(\Upsilon_{18})=\{2,3,6,9\}$ and $\Upsilon_{18}$ is the path $P_{4}: 2\sim 9 \sim 6 \sim 3$. By Lemma \ref{divisorgraph}, we have
$$\Gamma(\mathbb{Z}_{18})=\Upsilon_{18}[\Gamma(A_{2}),\Gamma(A_{3}),\Gamma(A_{6}),\Gamma(A_{9})],$$
where $\Gamma(A_{2})=\overline{K}_{6}$, $\Gamma(A_{3})=\overline{K}_{2}$, $\Gamma(A_{6})=K_{2}$ and $\Gamma(A_{9})$ is an isolated vertex. In Figure 1, the dotted lines between two circles mean that each vertex in one circle is adjacent to every vertex in the other circle.
\begin{figure}[!h]
\centering
\includegraphics[width=3.0in]{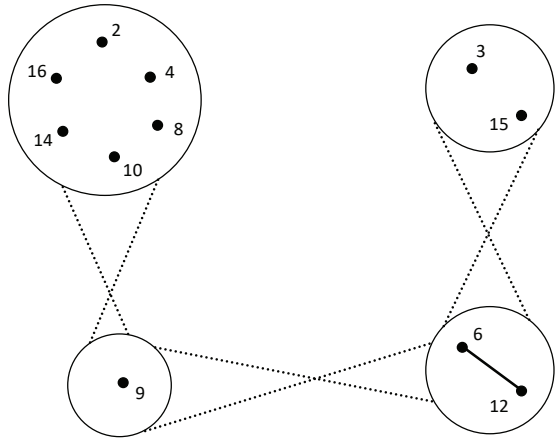}
\caption{Zero divisor graph of $\mathbb{Z}_{18}$}
\end{figure}
\end{example}

\section{Laplacian Spectrum}\label{sec-L-spectrum}

For a vertex $v$ of a graph $G$, $N_G(v)$ denotes the {\it neighbourhood} of $v$ in $G$, that is, the set of vertices of $G$ which are adjacent to $v$ in $G$.

\subsection{Laplacian spectrum of generalized join graphs}\label{sec-join-graph}

The following theorem was proved in \cite[Theorem 8]{cardoso} by Cardoso et al., in which the Laplacian spectrum of a generalized join graph $G[H_{1}, H_{2}, \ldots , H_{k}]$ is expressed in terms of the Laplacian spectrum of the graphs $H_{i}$ and the spectrum of another $k\times k$ matrix.

\begin{theorem}\cite{cardoso}\label{card}
Let $G$ be a graph on $k$ vertices with $V(G)=\{v_1,v_2,\ldots ,v_k\}$ and let $H_{1}, H_{2}, \ldots , H_{k}$ be $k$ pairwise disjoint graphs on $m_{1},m_{2}, \ldots , m_{k}$ vertices, respectively. Then the Laplacian spectrum of $G[H_{1}, H_{2}, \ldots , H_{k}]$ is given by
\begin{equation}\label{lap-spec}
\sigma_{L} \left(G\left[H_{1}, H_{2}, \ldots , H_{k}\right]\right)=\left(\underset{j=1}{\overset {k}\bigcup}\left(M_{j}+\left(\sigma _{L} \left(H_{j}\right)\setminus \{0\}\right)\right)\right) \bigcup \sigma (C),
\end{equation}
where
\begin{equation*}
M_{j}=
\begin{cases}
\underset{v_i\sim v_j}\sum m_{i} &\text{if $N_{G}(v_j)\neq \emptyset$}\\
0 &\text{otherwise,}
\end{cases}
\end{equation*}
\begin{equation*}
C=\begin{pmatrix}
                M_{1} & -s_{1,2} & \cdots & -s_{1,k} \\
                -s_{1,2} & M_{2} & \cdots & -s_{2,k} \\
                \cdots & \cdots & \cdots & \cdots \\
                -s_{1,k} & -s_{2,k} & \cdots & M_{k} \\
                \end{pmatrix},
\end{equation*}
and
\begin{equation*}
s_{i,j}=
\begin{cases}
\sqrt{m_{i}m_{j}} &\text{if $v_i\sim v_j$ in $G$}\\
0 &\text{otherwise.}
\end{cases}
\end{equation*}
In (\ref{lap-spec}), $M_{j}+(\sigma _{L} (H_{j})\setminus \{0\})$ means that $M_j$ is added to each element of $\sigma _{L} (H_{j})\setminus \{0\}$.
\end{theorem}

Consider $G$ as a vertex weighted graph by assigning the weight $m_i=|V(H_i)|$ to the vertex $v_i$ of $G$ for $1\leq i\leq k$.
Let $\textbf{L}(G)=(\textbf{l}_{i,j})$ be the $k\times k$ matrix, where
$$\textbf{l}_{i,j}=
\begin{cases}
-m_{j}  &\text{if $i\neq j$ and $v_i\sim v_j$}\\
\underset{v_i\sim v_r}\sum m_{r} &\text{if $i=j$}\\
0 &\text{otherwise.}
\end{cases}$$
The matrix $\textbf{L}(G)$ is called a {\it vertex weighted Laplacian matrix} of $G$, which is a zero row sum matrix but not symmetric in general.

Note that the matrix $C$ in Theorem \ref{card} is precisely the matrix $\mathcal{L}(G)$ defined in \cite[p. 317]{chung}, which is symmetric but need not be a zero row sum matrix. Further, if $W$ is the $k\times k$ diagonal matrix with diagonal entries $m_{1},m_{2}, \ldots , m_{k}$, then $\textbf{L}(G)=W^{-\frac{1}{2}}CW^{\frac{1}{2}}$
and so $\textbf{L}(G)$ and $C$ are similar. We thus have the following.

\begin{proposition}\label{C-L(G)}
$\sigma(C)= \sigma(\textbf{L}(G))$.
\end{proposition}

\subsection{Laplacian Spectrum of $\Gamma(\mathbb{Z}_{n})$}

Let $d_1,d_2,\ldots, d_k$ be the proper divisors of $n$. For $1\leq i\leq k$, we assign the weight $\phi \left(\frac{n}{d_{i}}\right)=| A_{d_{i}}|$ to the vertex $d_i$ of the graph $\Upsilon_n$. Define
\begin{equation}\label{m-d-j}
M_{d_{j}}=\sum\limits_{d_{i}\in N_{\Upsilon_n}(d_{j})}\phi \left(\frac{n}{d_{i}}\right)
\end{equation}
for $1\leq j\leq k$. The $k\times k$ vertex weighted Laplacian matrix $\textbf{L}(\Upsilon_n)$ of $\Upsilon_n$ defined in Section \ref{sec-join-graph} is given by
\begin{equation*}
\textbf{L}(\Upsilon_n)=\begin{pmatrix}
                M_{d_{1}} & -t_{1,2} & \cdots & -t_{1,k} \\
                -t_{2,1} & M_{d_{2}} & \cdots & -t_{2,k} \\
                \cdots & \cdots & \cdots & \cdots \\
                -t_{k,1} & -t_{k,2} & \cdots & M_{d_{k}} \\
                \end{pmatrix},
\end{equation*}
where
$$t_{i,j}=
\begin{cases}
\phi\left(\frac{n}{d_{j}}\right) &\text{if $d_{i}\sim d_{j}$ in $\Upsilon_n $}\\
0 &\text{otherwise}
\end{cases}$$
for $1\leq i\neq j\leq k$. The following theorem describes the Laplacian spectrum of the zero-divisor graph of $\mathbb{Z}_{n}$.

\begin{theorem}\label{main}
If $d_{1},d_{2},\cdots ,d_{k}$ are the proper divisors of $n$, then the Laplacian spectrum of $\Gamma(\mathbb{Z}_{n})$ is given by
\begin{eqnarray*}
\sigma _{L} \left(\Gamma\left(\mathbb{Z}_{n}\right)\right) = \underset{j=1}{\overset{k} \bigcup} \left(M_{d_{j}}+\left(\sigma _{L} \left(\Gamma\left(A_{d_{j}}\right)\right)\setminus \{0\}\right)\right)\bigcup  \sigma \left(\textbf{L}\left(\Upsilon_n\right)\right),
\end{eqnarray*}
where $M_{d_{j}}+\left(\sigma _{L} \left(\Gamma\left(A_{d_{j}}\right)\right)\right)$ means that $M_{d_j}$ is added to each element of $\sigma _{L} \left(\Gamma\left(A_{d_{j}}\right)\right)\setminus\{0\}$.
\end{theorem}

\begin{proof}
By Lemma \ref{divisorgraph}, we have $\Gamma(\mathbb{Z}_{n})=\Upsilon_{n}[\Gamma(A_{d_1}),\Gamma(A_{d_2}),\ldots, \Gamma(A_{d_k})]$. Then the result follows from Theorem \ref{card} and Proposition \ref{C-L(G)}.
\end{proof}

The Laplacian spectrum of the complete graph $K_m$ on $m$ vertices and its complement graph $\overline{K}_m$ are known. Indeed,
$$\sigma_L(K_m)=\left\{\begin{matrix}
                0 & m \\
               1 & m-1 \\
                \end{matrix}\right\}
                \text{ and }
                \sigma_L(\overline{K}_m)=\left\{\begin{matrix}
                0 \\
               m \\
                \end{matrix}\right\}.$$
By Corollary \ref{struc-coro}(i), $\Gamma(A_{d_{i}})$ is either $K_{\phi\left(\frac{n}{d_{i}}\right)}$ or $\overline{K}_{\phi\left(\frac{n}{d_{i}}\right)}$ for $1\leq i\leq k$. Also, $M_{d_j}>0$ as $\Upsilon_n$ is connected by Lemma \ref{UP}. Thus, by Theorem \ref{main}, out of the $n-\phi(n)-1$ number of Laplacian eigenvalues of $\Gamma(\mathbb{Z}_{n})$, $n-\phi(n)-1-k$ of them are known to be nonzero integer values. The remaining $k$ Laplacian eigenvalues of $\Gamma(\mathbb{Z}_{n})$ will come from the spectrum of $\textbf{L}(\Upsilon_n)$.

\begin{example}\label{lapspec-dis}
We discuss the Laplacian spectrum of $\Gamma(\mathbb{Z}_{n})$ for $n\in\{pq,p^2 q\}$, where $p$ and $q$ are distinct primes.\\

(i) Let $n=pq$, where $p<q$ are distinct primes. The proper divisors of $n$ are $p$ and $q$. So $\Upsilon_{pq}:p\sim q$ is $K_{2}$ and by Lemma \ref{divisorgraph}, $\Gamma(\mathbb{Z}_{pq})=\Upsilon_{pq}[\Gamma(A_{p}),\Gamma(A_{q})]$. By Corollary \ref{struc-coro}(i), $\Gamma(A_{p})=\overline{K}_{\phi(q)}$ and $\Gamma(A_{q})=\overline{K}_{\phi(p)}$.
We have $M_{p}=\phi\left(\frac{n}{q} \right)=\phi(p)$ and $M_{q}=\phi\left(\frac{n}{p} \right)=\phi(q)$. So, by Theorem \ref{main}, the Laplacian spectrum of $\Gamma(\mathbb{Z}_{pq})$ is given by
\begin{eqnarray*}
\sigma _{L} (\Gamma(\mathbb{Z}_{pq}))& =& (M_{p}+(\sigma _{L} (\Gamma(A_{p}))\setminus \{0\}))\bigcup (M_{q}+(\sigma _{L} (\Gamma(A_{q}))\setminus \{0\}))\bigcup \sigma (\textbf{L}(\Upsilon_{pq}))\\
&=& \left\{\begin{matrix}
 p-1 & q-1 \\
 q-2 & p-2 \\
 \end{matrix}\right\} \bigcup \sigma (\textbf{L}(\Upsilon_{pq})).
\end{eqnarray*}
We have $$\textbf{L}(\Upsilon_{pq})=\begin{pmatrix}
                \phi(p) & -\phi(p) \\
                -\phi(q) & \phi(q) \\
                \end{pmatrix}
                =\begin{pmatrix}
                p-1 & -(p-1) \\
                -(q-1) & q-1\\
                \end{pmatrix}.$$
which has eigenvalues $p+q-2$ and $0$. Thus the Laplacian spectrum of $\Gamma(\mathbb{Z}_{pq})$ is
    $$\sigma _{L} (\Gamma(\mathbb{Z}_{pq}))=\left\{\begin{matrix}
                 p-1 & q-1 & 0 & p+q-2 &\\
                q-2 & p-2 & 1 & 1 \\
                \end{matrix}\right\}.$$

Note that $\Gamma(\mathbb{Z}_{pq})=\Upsilon_{pq}[\Gamma(A_{p}),\Gamma(A_{q})]=\overline{K}_{\phi(p)}\vee \overline{K}_{\phi(q)}=K_{\phi(p),\phi(q)}$. Using the result known for the Laplacian eigenvalues of complete bipartite graphs, the Laplacian spectrum of $\Gamma(\mathbb{Z}_{pq})$ can also be obtained as above.\\

(ii) Let $n=p^{2}q$, where $p$ and $q$ are distinct primes. The proper divisors of $n$ are $p$, $q$, $pq$ and $p^2$. So $\Upsilon_{p^2q}$ is the path $P_{4}: p\sim pq \sim p^{2} \sim q$. By Lemma \ref{divisorgraph},
 $$\Gamma\left(\mathbb{Z}_{p^2q}\right)=\Upsilon_{p^2q}\left[\Gamma\left(A_{p}\right),\Gamma\left(A_{pq}\right),\Gamma\left(A_{p^{2}}\right),\Gamma\left(A_{q}\right)\right].$$
By Corollary \ref{struc-coro}(i), $\Gamma\left(A_{p}\right)=\overline{K}_{\phi(pq)}$, $\Gamma\left(A_{pq}\right)=K_{\phi(p)}$, $\Gamma\left(A_{p^{2}}\right)=\overline{K}_{\phi(q)}$ and $\Gamma\left(A_{q}\right)=\overline{K}_{\phi\left(p^{2}\right)}$.
We have
\begin{align*}
M_{p}&=\phi\left(\frac{n}{pq} \right)=\phi(p),\\
M_{pq}&=\phi\left(\frac{n}{p} \right)+\phi\left(\frac{n}{p^2} \right)=\phi(pq)+\phi(q),\\
M_{p^{2}}&= \phi\left(\frac{n}{pq} \right)+\phi\left(\frac{n}{q} \right)=\phi(p)+\phi(p^{2}),\\
M_{q}&=\phi\left(\frac{n}{p^2} \right)=\phi(q).
\end{align*}
So, by Theorem \ref{main}, the Laplacian spectrum of $\Gamma(\mathbb{Z}_{p^2q})$ is given by
\begin{eqnarray*}
\sigma_{L}\left(\Gamma\left(\mathbb{Z}_{p^2q}\right)\right)&=& \left(M_{p}+\left(\sigma _{L} \left(\Gamma\left(A_{p}\right)\right)\setminus \{0\}\right)\right)\bigcup \left(M_{pq}+\left(\sigma _{L} \left(\Gamma\left(A_{pq}\right)\right)\setminus \{0\}\right)\right) \\
& & \bigcup \left(M_{p^{2}}+\left(\sigma _{L} \left(\Gamma\left(A_{p^{2}}\right)\right)\setminus \{0\}\right)\right)\bigcup \left(M_{q}+\left(\sigma _{L} \left(\Gamma\left(A_{q}\right)\right)\setminus \{0\}\right)\right)\\
& & \bigcup \sigma \left(\textbf{L}\left(\Upsilon_{p^2q}\right)\right)\\
&=& \left\{\begin{matrix}
 p-1 & pq-1 & p^{2}-1 & q-1\\
 \phi(pq)-1 & \phi(p)-1 &  \phi(q)-1 & \phi(p^{2})-1\\
 \end{matrix}\right\}  \bigcup \sigma \left(\textbf{L}\left(\Upsilon_{p^{2}q}\right)\right).
\end{eqnarray*}
We have
$$\textbf{L}\left(\Upsilon_{p^{2}q}\right)=\begin{pmatrix}
                \phi(p) & -\phi(p) &  0 & 0 \\
                -\phi(pq) & \phi(pq)+\phi(q) &  -\phi(q) & 0 \\
                 0 & -\phi(p) &  \phi(p)+\phi(p^{2}) & -\phi(p^{2}) \\
                0 & 0 &  -\phi(q) & \phi(q) \\
                \end{pmatrix}.$$
The characteristic polynomial $Q(x)$ of $\textbf{L}\left(\Upsilon_{p^{2}q}\right)$ is $$x[x^3-(p^2+pq+q-3)x^2+((pq-1)(p^2+q-2)+(p-1)(q-1))x-(p-1)(q-1)(p^2+pq-p-1)].$$
If $K(x)=\frac{Q(x)}{x}$, then the algebraic connectivity and the Lapacian spectral radius of $\Gamma\left(\mathbb{Z}_{p^2q}\right)$ are the smallest and the largest roots of $K(x)=0,$ respectively. This follows from Theorem \ref{alg} in the last section.
\end{example}

\section{Laplacian Integrality of $\Gamma\left(\mathbb{Z}_{p^m}\right)$}\label{sec-integrality}

Recall that a graph $G$ is called Laplacian integral if all the Laplacian eigenvalues of $G$ are integers. The following proposition is an immediate consequence of the observation made in the paragraph after Theorem \ref{main}.

\begin{proposition}\label{integral}
The zero-divisor graph $\Gamma(\mathbb{Z}_{n})$ is Laplacian integral if and only if all the eigenvalues of $\textbf{L}(\Upsilon_n)$ are integers.
\end{proposition}

By Example \ref{lapspec-dis}(i), the graph $\Gamma(\mathbb{Z}_{pq})$ is Laplacian integral for distinct primes $p$ and $q$. In this section, we shall prove that $\Gamma(\mathbb{Z}_{p^t})$ is Lapacian integral for every prime $p$ and $t\geq 2$. One approach is to show that all the eigenvalues of $\textbf{L}\left(\Upsilon_{p^t}\right)$ are integers and then to use Proposition \ref{integral}. However, if $t$ is large, then it is more difficult to find the eigenvalues of $\textbf{L}\left(\Upsilon_{p^t}\right)$. We shall adopt a different approach to find the Laplacian characteristic polynomial of $\Gamma\left(\mathbb{Z}_{p^t}\right)$. For this, we first express $\Gamma\left(\mathbb{Z}_{p^t}\right)$ as the union and join of certain complete graphs and null graphs and then use Theorem \ref{join} below to find the Laplacian eigenvalues of $\Gamma\left(\mathbb{Z}_{p^t}\right)$.

For a graph $G$, we denote the characteristic polynomial of $L(G)$ by $\Theta (G,x)$. The following theorem gives the Laplacian characteristic polynomial of the join of two graphs, see \cite[Corollary 3.7]{mohar}.

\begin{theorem}\cite{mohar}\label{join}
Let $G_{1}$ and $G_{2}$ be two vertex disjoint graphs on $n_{1}$ and $n_{2}$ vertices, respectively. Then the Laplacian characteristic polynomial of $G_{1}\vee G_{2}$ is given by
$$\Theta (G_{1}\vee G_{2},x) = \frac{x(x - n_{1} - n_{2})}{(x - n_{1})(x - n_{2})}\Theta(G_{1}, x - n_{2})\Theta(G_{2}, x - n_{1}).$$
\end{theorem}

\begin{theorem}\label{lapspec-primepower}
Let $n=p^t$ where $p$ is a prime and $t\geq 2$ is a positive integer. Then the following hold.
\begin{enumerate}
\item[(i)] If $t=2$, then the Laplacian spectrum of $\Gamma(\mathbb{Z}_{n})$ is given by
$$\left\{\begin{matrix}
                0 \\
                1 \\
                \end{matrix}\right\}  \mbox{ or } \left\{\begin{matrix}
                 p-1 & 0 \\
                p-2 & 1 \\
                \end{matrix}\right\}$$
according as $p=2$ or $p\geq 3$.
\item[(ii)]  If $t=2m$ for some integer $m\geq 2$, then the Laplacian spectrum of $\Gamma(\mathbb{Z}_{n})$ is given by
$$\left\{\begin{matrix}
                p^{2m-1}-1 & p^{2m-2}-1 & \cdots & p^{m+1}-1 & p^{m}-1 & p^{m-1}-1 & \cdots & p-1 & 0 \\
                \phi(p) & \phi(p^{2}) & \cdots & \phi(p^{m-1}) & \phi(p^{m})-1 & \phi(p^{m+1}) & \cdots & \phi(p^{2m-1}) & 1 \\
                \end{matrix}\right\}.$$
\item[(iii)] If $t=2m+1$ for some integer $m\geq 1$, then the Laplacian spectrum of $\Gamma(\mathbb{Z}_{n})$ is given by
$$\left\{\begin{matrix}
                p^{2m}-1 & p^{2m-1}-1 & \cdots & p^{m+1}-1 & p^{m}-1 & p^{m-1}-1 & \cdots &  p-1 & 0 \\
                \phi(p) & \phi(p^{2}) & \cdots & \phi(p^{m}) & \phi(p^{m+1})-1 & \phi(p^{m+2}) & \cdots & \phi(p^{2m}) & 1 \\
                \end{matrix}\right\}.$$
\end{enumerate}
\end{theorem}

\begin{proof}
(i) We have $\Gamma\left(\mathbb{Z}_{p^2}\right)=\Gamma\left(A_p\right)$ is the complete graph $K_{p-1}$ by Corollary \ref{struc-coro}(i) and so the results follows depending on $p=2$ or not.\\

(ii) Here $n=p^{2m}$ with $m\geq 2$ and the proper divisors of $n$ are $p,p^{2},\ldots, p^{2m-1}$. We shall express the graph $\Upsilon_{p^{2m}}$ as the join and union of certain graphs. Observe that the vertex $p^{i}$, $1\leq i\leq 2m-1$, of $\Upsilon_{p^{2m}}$ is adjacent to the vertex $p^j$ for every $j\geq 2m-i$ with $j\neq i$. Define the following graphs $H_1,H_2,\ldots, H_m$ recursively, where $\{x\}$ denotes the graph with one vertex $x$:
\begin{eqnarray*}
H_{1} &=& \left\{p^{m}\right\} \\
H_{2} &=& \left\{p^{m+1}\right\}\vee \left[\left\{p^{m-1}\right\}\cup H_{1} \right] \\
H_{3} &=& \left\{p^{m+2}\right\}\vee\left[\left\{p^{m-2}\right\}\cup H_{2}\right] \\
 &\vdots &  \\
H_{m} &=& \left\{p^{2m-1}\right\}\vee\left[\{p\}\cup H_{m-1}\right].
\end{eqnarray*}
It can be seen that $H_m$ is precisely the graph $\Upsilon_{p^{2m}}$. Now define the graphs $G_1,G_2,\ldots, G_m$ recursively as given below:
\begin{eqnarray*}
G_{1} &=& K_{\phi\left(p^{m}\right)} \\
G_{2} &=& K_{\phi\left(p^{m-1}\right)}\vee\left[\overline{K}_{\phi\left(p^{m+1}\right)}\cup G_{1}\right] \\
G_{3} &=& K_{\phi\left(p^{m-2}\right)}\vee\left[\overline{K}_{\phi\left(p^{m+2}\right)}\cup G_{2}\right] \\
 &\vdots &  \\
G_{m} &=& K_{\phi(p)}\vee\left[\overline{K}_{\phi\left(p^{2m-1}\right)}\cup G_{m-1}\right].
\end{eqnarray*}
We have $\Gamma\left(\mathbb{Z}_{p^{2m}}\right)=\Upsilon_{p^{2m}}\left[\Gamma\left(A_p\right),\Gamma\left(A_{p^2}\right),\ldots,\Gamma\left(A_{p^{2m-1}}\right)\right]$. Since $\Gamma\left(A_{p^j}\right)=K_{\phi\left(p^{2m-j}\right)}$ for $m\leq j\leq 2m-1$ and $\Gamma\left(A_{p^j}\right)=\overline{K}_{\phi\left(p^{2m-j}\right)}$ for $1\leq j\leq m-1$, it follows that $\Gamma(\mathbb{Z}_{p^{2m}})$ is precisely the graph $G_m$. The Laplacian characteristic polynomial of $G_1$ is
$$\Theta \left(G_{1},x\right)=x\times\left[x-\phi\left(p^{m}\right)\right]^{\phi\left(p^{m}\right)-1}$$
and that of $\overline{K}_{\phi\left(p^{m+1}\right)}\cup G_{1}$ is
$$\Theta \left(\overline{K}_{\phi\left(p^{m+1}\right)}\cup G_{1},x\right)=x^{\phi\left(p^{m+1}\right)+1}\times \left[x-\phi\left(p^{m}\right)\right]^{\phi\left(p^{m}\right)-1}.$$
Using Theorem \ref{join}, the Laplacian characteristic polynomial of $G_2$ is
\begin{eqnarray*}
\Theta \left(G_{2},x\right) &= & x \times \left[x-\phi\left(p^{m-1}\right)-\phi\left(p^{m}\right)-\phi\left(p^{m+1}\right)\right]^{\phi\left(p^{m-1}\right)}\\
& & \times \left[x-\phi\left(p^{m-1}\right)-\phi\left(p^{m}\right)\right]^{\phi\left(p^{m}\right)-1}\times \left[x-\phi\left(p^{m-1}\right)\right]^{\phi\left(p^{m+1}\right)}.
\end{eqnarray*}
Now the Laplacian characteristic polynomial of $\overline{K}_{\phi\left(p^{m+2}\right)}\cup G_{2}$ is
\begin{eqnarray*}
\Theta \left(\overline{K}_{\phi\left(p^{m+2}\right)}\cup G_{2},x\right) &=& x^{\phi\left(p^{m+2}\right)+1}\times \left[x-\phi\left(p^{m-1}\right)-\phi\left(p^{m}\right)-\phi\left(p^{m+1}\right)\right]^{\phi\left(p^{m-1}\right)}\\
& & \times \left[x-\phi\left(p^{m-1}\right)-\phi\left(p^{m}\right)\right]^{\phi\left(p^{m}\right)-1}\times\left[x-\phi\left(p^{m-1}\right)\right]^{\phi\left(p^{m+1}\right)}.
\end{eqnarray*}
Again using Theorem \ref{join}, it can be calculated that the Laplacian characteristic polynomial of $G_3$ is
\begin{eqnarray*}
\Theta \left(G_{3},x\right) &= &
x\times \left[x-\phi\left(p^{m-2}\right)-\phi\left(p^{m-1}\right)-\phi\left(p^{m}\right)-\phi\left(p^{m+1}\right)-\phi\left(p^{m+2}\right)\right]^{\phi\left(p^{m-2}\right)}\\
& & \times \left[x-\phi\left(p^{m-2}\right)-\phi\left(p^{m-1}\right)-\phi\left(p^{m}\right)-\phi\left(p^{m+1}\right)\right]^{\phi\left(p^{m-1}\right)}\\
& & \times \left[x-\phi\left(p^{m-2}\right)-\phi\left(p^{m-1}\right)-\phi\left(p^{m}\right)\right]^{\phi\left(p^{m}\right)-1}\\
& & \times \left[x-\phi\left(p^{m-2}\right)-\phi\left(p^{m-1}\right)\right]^{\phi\left(p^{m+1}\right)}\\
& & \times\left[x-\phi\left(p^{m-2}\right)\right]^{\phi\left(p^{m+2}\right)}.
\end{eqnarray*}
Continuing in this way, we finally get that
\begin{eqnarray*}
\Theta \left(G_{m},x\right) &=& x\times \left(x-\sum\limits_{i=1}^{2m-1} \phi\left(p^{i}\right)\right)^{\phi(p)}\times\cdots \times\left(x-\sum\limits_{i=1}^{m+1} \phi \left(p^{i}\right)\right)^{\phi\left(p^{m-1}\right)}\\
& & \times \left(x-\sum\limits_{i=1}^{m} \phi\left(p^{i}\right)\right)^{\phi\left(p^{m}\right)-1}\times \left(x-\sum\limits_{i=1}^{m-1} \phi\left(p^{i}\right)\right)^{\phi\left(p^{m+1}\right)} \\
& & \times \cdots \times \left(x-\sum\limits_{i=1}^{2} \phi \left(p^{i}\right)\right)^{\phi\left(p^{2m-2}\right)}\times \left(x- \phi (p)\right)^{\phi\left(p^{2m-1}\right)}.
\end{eqnarray*}
Since $\Gamma(\mathbb{Z}_{p^{2m}})=G_m$, we have $\Theta \left(\Gamma\left(\mathbb{Z}_{p^{2m}}\right),x\right)=\Theta \left(G_{m},x\right)$. Then the result follows from the above using the fact that
$\underset{i=1}{\overset{r}\sum} \phi\left(p^i\right)=p^r -1$ for any positive integer $r$.\\

(ii) Here $n=p^{2m+1}$ and the proper divisors of $n$ are $p,p^{2},\ldots, p^{2m}$. As in (i), we shall express the graph $\Upsilon_{p^{2m+1}}$ as the join and union of certain graphs.
The vertex $p^{i}$, $1\leq i\leq 2m$, of $\Upsilon_{p^{2m+1}}$ is adjacent to the vertex $p^j$ for every $j\geq 2m+1-i$. Define the graphs $X_1,X_2,\ldots, X_m$ recursively as given below:
\begin{eqnarray*}
X_{1} &=& \left\{p^{m+1}\right\}\vee \left\{p^{m}\right\} \\
X_{2} &=& \left\{p^{m+2}\right\}\vee \left[\left\{p^{m-1}\right\}\cup X_{1} \right] \\
X_{3} &=& \left\{p^{m+3}\right\}\vee\left[\left\{p^{m-2}\right\}\cup X_{2}\right] \\
 &\vdots &  \\
X_{m} &=& \left\{p^{2m}\right\}\vee\left[\{p\}\cup X_{m-1}\right].
\end{eqnarray*}
Then $X_m$ is precisely the graph $\Upsilon_{p^{2m+1}}$. Now define the graphs $Y_1,Y_2,\ldots, Y_m$ recursively as given below:
\begin{eqnarray*}
Y_{1} &=& K_{\phi\left(p^{m}\right)}\vee \overline{K}_{\phi\left(p^{m+1}\right)} \\
Y_{2} &=& K_{\phi\left(p^{m-1}\right)}\vee\left[\overline{K}_{\phi\left(p^{m+2}\right)}\cup Y_{1}\right] \\
Y_{3} &=& K_{\phi\left(p^{m-2}\right)}\vee\left[\overline{K}_{\phi\left(p^{m+3}\right)}\cup Y_{2}\right] \\
 &\vdots &  \\
Y_{m} &=& K_{\phi(p)}\vee\left[\overline{K}_{\phi\left(p^{2m}\right)}\cup Y_{m-1}\right].
\end{eqnarray*}
As in (i), it can be seen that $\Gamma\left(\mathbb{Z}_{p^{2m+1}}\right)$ is precisely the graph $Y_m$.
Using Theorem \ref{join}, we get
$$\Theta\left(Y_{1},x\right)=x\times \left[x-\phi\left(p^m\right)-\phi\left(p^{m+1}\right)\right]^{\phi\left(p^m\right)}\times \left[x-\phi\left(p^m\right)\right]^{\phi\left(p^{m+1}\right)-1}.$$
Starting with $\Theta\left(Y_{1},x\right)$ and applying the argument as in (i), we can calculate the Laplacian characteristic polynomials of $Y_2,Y_3,\ldots ,Y_{m}$ and get the required result.
\end{proof}

As a consequence of Proposition \ref{integral} and Theorem \ref{lapspec-primepower}, we have the following.
\begin{corollary}
If $p$ is a prime and $t\geq 2$, then $\Gamma\left(\mathbb{Z}_{p^t}\right)$ is Lapacian integral and so all the eigenvalues of $\textbf{L}\left(\Upsilon_{p^t}\right)$ are integers.
\end{corollary}

\begin{corollary}\label{coro-lapspec}
Let $n=p^{t}$ for some prime $p$ and positive integer $t$ with $n\neq 4$. Then $\lambda(\Gamma(\mathbb{Z}_{p^t}))=\vert\Gamma(\mathbb{Z}_{p^t})\vert$.
\end{corollary}

\begin{proof}
We have $\left\vert\Gamma\left(\mathbb{Z}_{p^t}\right)\right\vert=p^t-\phi(p^t)-1=p^{t-1}-1$. From Theorem \ref{lapspec-primepower}, we get that $\lambda(\Gamma(\mathbb{Z}_{p^t}))=p^{t-1}-1$ and so the corollary follows.
\end{proof}

\section{Algebraic connectivity and Laplacian spectral radius of $\Gamma(\mathbb{Z}_{n})$}\label{sec-alg-con}

In this section, we shall study the algebraic connectivity and the Laplacian spectral radius of $\Gamma(\mathbb{Z}_{n})$. We recall two well-known bounds for the Laplacian spectral radius of a graph.

\begin{theorem}\cite{fei}\label{spec-rad-equal}
If $G$ is a graph on $m$ vertices, then  $\lambda(G) \leq m$. Further, equality holds if and only if $\overline{G}$ is disconnected if and only if $G$ is the join of two graphs.
\end{theorem}

The above theorem follows from the relation $\lambda(G) = m -\mu(\overline{G})$ and the fact that $\overline{G}$ is disconnected if and only if $G$ is the join of two graphs. The following result was proved in \cite[Theorem 2.3]{zhang}.

\begin{theorem}\cite{zhang}\label{spec-rad-maxd}
Let $G$ be a connected graph on $m$ vertices with maximal degree $\Delta(G)$. Then $\lambda(G)\geq \Delta(G) +1$, and equality holds if and only if $\Delta(G) =m-1$.
\end{theorem}

The following proposition characterizes the values of $n$ for which the complement graph of $\Gamma(\mathbb{Z}_{n})$ is disconnected. Note that if $n=4$, then $\Gamma(\mathbb{Z}_{4})=\overline{\Gamma(\mathbb{Z}_{4})}=\{2\}$ is a singleton.

\begin{proposition}\label{com-disconn}
$\overline{\Gamma(\mathbb{Z}_{n})}$ is disconnected if and only if $n$ is a product of two distinct primes or $n$ is a prime power with $n\neq 4$.
\end{proposition}

\begin{proof}
If $n=pq$ for distinct primes $p$ and $q$, then $\Gamma(\mathbb{Z}_{n})=K_{\phi(p),\phi(q)}$, see Example \ref{lapspec-dis}(i). If $n=p^{2}$ for some prime $p\geq 3$, then $\Gamma(\mathbb{Z}_{n})=\Gamma(A_{p})=K_{\phi(p)}$ by Corollary \ref{struc-coro}(i) and it contains at least two vertices. If $n=p^t$ for some prime $p$ with $t \geq 3$, then the vertex $p^{t-1}$ is adjacent to all other vertices of $\Gamma(\mathbb{Z}_{n})$. In all the three cases, it follows that $\overline{\Gamma(\mathbb{Z}_{n})}$ is disconnected.

Conversely, let $n=p_1^{n_1}p_2^{n_2}\cdots p_r^{n_r}$, where $r$, $n_1,n_2,\ldots ,n_r$ are positive integers and $p_1,p_2,\ldots, p_r$ are distinct primes. Suppose that $r\geq 2$ and that $n_{1}>1$ or $n_{2}>1$ if $r=2$. We show that $\overline{\Gamma(\mathbb{Z}_{n})}$ is connected.

The vertices $p_i$ and $p_j$ are not adjacent in $\Upsilon_n$ for $1\leq i\neq j\leq k$. So the vertices $p_1,p_2,\ldots, p_{r}$ form a clique in $\overline{\Upsilon_n}$.

Let $v$ be vertex of $\overline{\Upsilon_n}$ different from $p_1,p_2,\ldots, p_{r}$. There exists $i\in\{1,2,\ldots,r\}$ such that $p_i^t$ divides $v$, but $p_i^{t+1}$ does not divide $v$ for some $t$ with $0\leq t < n_i$. Then, for $j\in\{1,2,\ldots,r\}\setminus\{i\}$, $v$ and $p_j$ are not adjacent in $\Upsilon_n$ as $n$ does not divide $vp_j$ and so $v$ and $p_j$ are adjacent in $\overline{\Upsilon_n}$.
It follows that $\overline{\Upsilon_n}$ is connected. If $d_1,d_2,\ldots, d_k$ are the proposer divisors of $n$, then
$\Gamma(\mathbb{Z}_{n})=\Upsilon_n\left[\Gamma(A_{d_{1}}),\Gamma(A_{d_{2}}),\ldots ,\Gamma(A_{d_{k}})\right]$ implies that   $\overline{\Gamma(\mathbb{Z}_{n})}=\overline{\Upsilon_n}\left[\overline{\Gamma(A_{d_{1}})},\overline{\Gamma(A_{d_{2}})},\ldots ,\overline{\Gamma(A_{d_{k}})}\right]$. As $k\geq 2$, Lemma \ref{connected} implies that $\overline{\Gamma(\mathbb{Z}_{n})}$ is connected.
\end{proof}

The following proposition characterizes the values of $n$ for which equality holds in Theorem \ref{spec-rad-equal} when $G=\Gamma(\mathbb{Z}_n)$.

\begin{proposition}\label{equal-m}
$\lambda(\Gamma(\mathbb{Z}_{n}))=\vert\Gamma(\mathbb{Z}_{n})\vert$ if and only if $n$ is a product of two distinct primes or $n$ is a prime power with $n\neq 4$.
\end{proposition}

\begin{proof}
If $n$ is not a product of two distinct primes nor a prime power, then $\overline{\Gamma(\mathbb{Z}_{n})}$ is connected by Proposition \ref{com-disconn}. In this case, $\lambda(\Gamma(\mathbb{Z}_{n}))< \vert\Gamma(\mathbb{Z}_{n})\vert$ by Theorem \ref{spec-rad-equal}. If $n=4$, then $\lambda(\Gamma(\mathbb{Z}_{4}))=0<1=\vert\Gamma(\mathbb{Z}_{4})\vert$.

If $n$ is a prime power with $n\neq 4$, then $\lambda(\Gamma(\mathbb{Z}_{n}))=\vert\Gamma(\mathbb{Z}_{n})\vert$ by Corollary \ref{coro-lapspec}.
Assume that $n=pq$ for two distinct primes $p$ and $q$. Then $\vert\Gamma(\mathbb{Z}_{pq})\vert=pq-\phi(pq)-1=p+q-2$. From Example \ref{lapspec-dis}(i), we have $\lambda(\Gamma(\mathbb{Z}_{pq}))=p+q-2$ and so $\lambda(\Gamma(\mathbb{Z}_{pq}))=\vert\Gamma(\mathbb{Z}_{pq})\vert$.
\end{proof}

The following theorem was proved in \cite[Theorem 3.2]{ju}, which determines the vertex connectivity $\kappa(\Gamma(\mathbb{Z}_{n}))$ of $\Gamma(\mathbb{Z}_{n})$.

\begin{theorem}\cite{ju}\label{ver-con}
Let $p$ be the smallest prime divisor of $n$ and let $\delta\left(\Gamma\left(\mathbb{Z}_{n}\right)\right)$ denote the minimal degree of $\Gamma\left(\mathbb{Z}_{n}\right)$. Then the following hold:
\begin{enumerate}[(i)]
\item If $n$ is divisible by at least two distinct primes, then $\kappa\left(\Gamma\left(\mathbb{Z}_{n}\right)\right)=\delta\left(\Gamma\left(\mathbb{Z}_{n}\right)\right)=p-1$ and the vertex $p$ has minimal degree.
\item Let $n=p^t$ with $t\geq 2$. Then $\kappa\left(\Gamma\left(\mathbb{Z}_{n}\right)\right)=\delta\left(\Gamma\left(\mathbb{Z}_{n}\right)\right)=p-2$ if $t=2$, and $\kappa\left(\Gamma\left(\mathbb{Z}_{n}\right)\right)=\delta\left(\Gamma\left(\mathbb{Z}_{n}\right)\right)=p-1$ if $t>2$. In both cases, the vertex $p$ has minimal degree.
\end{enumerate}
\end{theorem}

We shall use Theorem \ref{ver-con} along with the following result of Krikland et al. \cite[Theorem 2.1]{krik} to characterize the values of $n$ for which vertex connectivity and algebraic connectivity of $\Gamma(\mathbb{Z}_{n})$ are equal.

\begin{theorem}\cite{krik} \label{alg-con-equal}
Let $G$ be a noncomplete connected graph on $m$ vertices. Then $\kappa(G)=\mu(G)$ if and only if $G$ can be written as $G_{1}\vee G_{2}$, where $G_{1}$ is a disconnected graph on $m-\kappa(G)$ vertices and $G_{2}$ is a graph on $\kappa(G)$ vertices with $\mu(G_{2}) \geq 2\kappa(G)-m$.
\end{theorem}

\begin{proposition}\label{equal-algcon-lapspec}
$\mu(\Gamma(\mathbb{Z}_{n}))=\kappa(\Gamma(\mathbb{Z}_{n}))$ if and only if $n$ is product of two distinct primes or $n=p^t$ for some prime $p$ and integer $t\geq 3$.
\end{proposition}

\begin{proof}
We have $\mu(\Gamma(\mathbb{Z}_{n}))\leq \kappa(\Gamma(\mathbb{Z}_{n}))$ if and only if $\Gamma(\mathbb{Z}_{n})$ is not a complete graph, that is, if and only if $n$ is not the square of a prime by Corollary \ref{z-n-complete}.

If $n$ is not a product of two distinct primes nor a prime power, then $\overline{\Gamma(\mathbb{Z}_{n})}$ is connected by Proposition \ref{com-disconn} and so $\Gamma(\mathbb{Z}_{n})$ is not a join of two graphs.
Since $\Gamma(\mathbb{Z}_{n})$ is noncomplete and connected, Theorem \ref{alg-con-equal} implies that $\mu(\Gamma(\mathbb{Z}_{n}))< \kappa(\Gamma(\mathbb{Z}_{n}))$.

If $n=pq$ for some primes $p<q$, then $\kappa(\Gamma(\mathbb{Z}_{pq}))=p-1$ by Theorem \ref{ver-con}(i). From Example \ref{lapspec-dis}(i), we have $\mu(\Gamma(\mathbb{Z}_{pq}))=p-1$ and so $\mu(\Gamma(\mathbb{Z}_{pq}))=\kappa(\Gamma(\mathbb{Z}_{pq}))$.

If $n=p^t$ for some prime $p$ and positive integer $t\geq 3$, then $\mu(\Gamma(\mathbb{Z}_{p^t}))=p-1=\kappa(\Gamma(\mathbb{Z}_{p^t}))$ by Theorems \ref{lapspec-primepower}(ii), \ref{lapspec-primepower}(iii) and \ref{ver-con}(ii).
\end{proof}

\begin{theorem}\label{alg}
The following hold:
\begin{enumerate}[(i)]
\item If $n$ is not a prime power nor a product of two distinct primes, then $\mu(\Gamma(\mathbb{Z}_{n}))$ is the second smallest eigenvalue of $\textbf{L}(\Upsilon_n)$.
\item If $n$ is not a prime power, then $\lambda(\Gamma(\mathbb{Z}_{n}))$ is the largest eigenvalue of $\textbf{L}(\Upsilon_n)$.
\end{enumerate}
\end{theorem}

\begin{proof}
By Theorem \ref{main}, the Laplacian spectrum of $\Gamma(\mathbb{Z}_{n})$ is given by
\begin{eqnarray*}
\sigma _{L} \left(\Gamma\left(\mathbb{Z}_{n}\right)\right) = \underset{j=1}{\overset{k} \bigcup} \left(M_{d_{j}}+\left(\sigma _{L} \left(\Gamma\left(A_{d_{j}}\right)\right)\setminus \{0\}\right)\right)\bigcup  \sigma \left(\textbf{L}\left(\Upsilon_n\right)\right),
\end{eqnarray*}
where $d_{1},d_{2},\cdots ,d_{k}$ are the proper divisors of $n$ and $M_{d_j}$ is defined in (\ref{m-d-j}) for $1\leq j\leq k$.\\

(i) Let $p$ be the smallest prime divisor of $n$. Since $n$ is not a product of two distinct primes nor a prime power, Theorem \ref{ver-con}(i) and Proposition \ref{equal-algcon-lapspec} give that
\begin{equation}\label{lessthan-p-1}
\mu(\Gamma(\mathbb{Z}_{n}))< \kappa(\Gamma(\mathbb{Z}_{n}))=p-1,
\end{equation}
Let $\alpha$ be the minimum of the Laplacian eigenvalues of $\Gamma(\mathbb{Z}_{n})$ which are contained in
$$\underset{j=1}{\overset{k} \bigcup} \left(M_{d_{j}}+\left(\sigma _{L} \left(\Gamma\left(A_{d_{j}}\right)\right)\setminus \{0\}\right)\right).$$
Then
$$ \alpha=\min\left\{\mu\left(\Gamma\left(A_{d_{j}}\right)\right)+M_{d_{j}}: 1\leq j\leq k\right\},$$
where the minimum is taken over all $j$ for which $\Gamma(A_{d_{j}})$ is not a singleton.
The connectedness of $\Upsilon_n$ (Lemma \ref{UP}) implies that
$M_{d_{j}}\geq p-1$ for $1\leq j\leq k$ and hence $\alpha\geq p-1$. Then (\ref{lessthan-p-1}) implies that $\mu(\Gamma(\mathbb{Z}_{n}))$ must be an eigenvalue of $\textbf{L}(\Upsilon_n)$. Since $0$ is an eigenvalue of $\textbf{L}(\Upsilon_n)$, it follows that $\mu(\Gamma(\mathbb{Z}_{n}))$ is the second smallest eigenvalue of $\textbf{L}(\Upsilon_n)$. \\

(ii) If $n$ is a product of two distinct primes, then the result follows from Example \ref{lapspec-dis}(i). Assume that $n$ is not a prime power nor a product of two distinct primes. Then $\overline{\Gamma(\mathbb{Z}_{n})}$ is connected by Proposition \ref{com-disconn}. It follows from Theorems \ref{spec-rad-equal} and \ref{spec-rad-maxd} that
\begin{equation}\label{greater-than}
\lambda(\Gamma(\mathbb{Z}_{n}))> \Delta (\Gamma(\mathbb{Z}_{n}))+1,
\end{equation}
where $\Delta (\Gamma(\mathbb{Z}_{n}))$ is the maximal degree in $\Gamma(\mathbb{Z}_{n})$. Let $\beta$ be the maximum of the Laplacian eigenvalues of $\Gamma(\mathbb{Z}_{n})$ which are contained in $\underset{j=1}{\overset{k} \bigcup} \left(M_{d_{j}}+\left(\sigma _{L} \left(\Gamma\left(A_{d_{j}}\right)\right)\setminus \{0\}\right)\right)$. Then
$$ \beta=\max\left\{\lambda\left(\Gamma\left(A_{d_{j}}\right)\right)+M_{d_{j}}: 1\leq j\leq k\right\},$$
where the maximum is taken over all $j$ for which $\Gamma(A_{d_{j}})$ is not a singleton.

Let $v$ be a vertex of $\Gamma(\mathbb{Z}_{n})$. Then $v\in A_{d_{j}}$ for some $j\in\{1,2,\ldots, k\}$. By Corollary \ref{struc-coro}(i), $\Gamma\left(A_{d_j}\right)$ is   $K_{\phi\left(\frac{n}{d_j}\right)}$ or $\overline{K}_{\phi\left(\frac{n}{d_j}\right)}$.
If $\Gamma\left(A_{d_{j}}\right)=K_{\phi\left(\frac{n}{d_{j}}\right)}$, then $$deg(v)+1=M_{d_{j}}+\phi\left(\frac{n}{d_{j}}\right)=M_{d_{j}}+\lambda\left(\Gamma\left(A_{d_{j}}\right)\right).$$
If $\Gamma\left(A_{d_{j}}\right)=\overline{K}_{\phi\left(\frac{n}{d_{j}}\right)}$, then $$deg(v)+1=M_{d_{j}}+1>M_{d_{j}}=M_{d_{j}}+\lambda\left(\Gamma\left(A_{d_{j}}\right)\right).$$
Thus $\Delta \left(\Gamma\left(\mathbb{Z}_{n}\right)\right)+1=\max\{\deg(v)+1: v\in V\left(\Gamma\left(\mathbb{Z}_{n}\right)\right)\}\geq M_{d_{j}}+\lambda\left(\Gamma\left(A_{d_{j}}\right)\right)$ for $1\leq j\leq k$ and hence $\Delta \left(\Gamma\left(\mathbb{Z}_{n}\right)\right)+1\geq \beta$.
Then (\ref{greater-than}) gives that
$$\lambda(\Gamma(\mathbb{Z}_{n}))> \Delta (\Gamma(\mathbb{Z}_{n}))+1\geq \beta$$
and it follows that $\lambda(\Gamma(\mathbb{Z}_{n}))$ is the largest eigenvalue of $\textbf{L}(\Upsilon_n)$.
\end{proof}

From the proof of the above theorem, the following corollary follows.

\begin{corollary}
Let $d_{1},d_{2},\cdots ,d_{k}$ be the proper divisors of $n$. Then the following hold:
\begin{enumerate}[(i)]
\item If $n$ is not a prime power nor a product of two distinct primes, then $\mu(\Gamma(\mathbb{Z}_{n}))$ is not contained in $\left(M_{d_{j}}+\left(\sigma _{L} \left(\Gamma\left(A_{d_{j}}\right)\right)\setminus \{0\}\right)\right)$ for $1\leq i\leq k$.
\item If $n$ is not a prime power, then $\lambda(\Gamma(\mathbb{Z}_{n}))$ is not contained in $\left(M_{d_{j}}+\left(\sigma _{L} \left(\Gamma\left(A_{d_{j}}\right)\right)\setminus \{0\}\right)\right)$ for $1\leq i\leq k$.
\end{enumerate}
\end{corollary}

\noindent{\bf Addresses}:\\

\noindent 1) School of Mathematical Sciences,\\
National Institute of Science Education and Research (NISER), Bhubaneswar,\\
P.O.- Jatni, District- Khurda, Odisha - 752050, India\medskip

\noindent 2) Homi Bhabha National Institute (HBNI),\\
Training School Complex, Anushakti Nagar,\\
Mumbai - 400094, India\medskip

\noindent E-mails: sriparna@niser.ac.in, klpatra@niser.ac.in, bksahoo@niser.ac.in

\begin{thebibliography}{99}
\bibitem{akbari} S. Akbari and A. Mohammadian, On the zero-divisor graph of a commutative ring, J. Algebra 274 (2004), 847--855.

\bibitem{anderson} D. F. Anderson and P. S. Livingston, The zero-divisor graph of a commutative ring, J. Algebra 217 (1999) 434--447.

\bibitem{beck} I. Beck, Coloring of commutative rings, J. Algebra 116 (1988), 208--226.

\bibitem{cardoso} D. M. Cardoso, M. A. A. de Freitas, E. A. Martins and M. Robbiano, Spectra of graphs obtained by a generalization of the join graph operation, Discrete Math. 313 (2013), 733--741.

\bibitem{chung} F. R. K. Chung and R. P. Langlands, A combinatorial Laplacian with vertex weights, J. Combin. Theory Ser. A 75 (1996), 316--327.

\bibitem{cvet} D. Cvetkovi\'{c}, P. Rowlinson and S. Simi\'{c}, An Introduction to the Theory of Graph Spectra, London Mathematical Society Student Texts, 75, Cambridge University Press, Cambridge, 2010.

\bibitem{fei} M. Fiedler, Algebraic connectivity of graphs, Czechoslovak Math. J. 23(98) (1973), 298--305.

\bibitem{ju} T. Ju and M. Wu, On iteration digraph and zero-divisor graph of the ring $\mathbb{Z}_n$, Czechoslovak Math. J. 64(139) (2014), 611–-628.

\bibitem{krik} S. J. Kirkland, J. J. Molitierno, M. Neumann and B. L. Shader, On graphs with equal algebraic and vertex connectivity, Linear Algebra Appl. 341 (2002) 45–-56.

\bibitem{mohar} B. Mohar, The Laplacian spectrum of graphs. in - Graph Theory, Combinatorics, and Applications. Vol. 2 (Kalamazoo, MI, 1988), 871--898, Wiley-Intersci. Publ., Wiley, New York, 1991.

\bibitem{schwenk} A. J. Schwenk, Computing the characteristic polynomial of a graph, in - Graphs and Combinatorics, pp. 153--172, Lecture Notes in Math., Vol. 406, Springer, Berlin, 1974.

\bibitem{mathew} M. Young, Adjacency matrices of zero-divisor graphs of integers modulo $n$, Involve 8 (2015), 753--761.

\bibitem{zhang} X.-D. Zhang and R. Luo, The spectral radius of triangle-free graphs, Australas. J. Combin. 26 (2002), 33--39.
\end{thebibliography}
\end{document}